\documentclass[12pt]{amsart}
\usepackage{graphicx,verbatim,amssymb,amsmath,multirow,bbold}
\usepackage[all]{xy}
\usepackage[active]{srcltx}
\usepackage{hyperref}

\DeclareSymbolFont{bbold}{U}{bbold}{m}{n}
\DeclareSymbolFontAlphabet{\mathbbold}{bbold}

\newtheorem{thm}{Theorem}[section]
\newtheorem{cor}[thm]{Corollary}
\newtheorem{lem}[thm]{Lemma}
\newtheorem{prop}[thm]{Proposition}

\theoremstyle{definition}
\newtheorem{dfn}[thm]{Definition}
\theoremstyle{remark}
\newtheorem{rem}[thm]{Remark}
\newtheorem{ex}[thm]{Example}
\numberwithin{equation}{section}

\def\Q{\mathbb{Q}}
\def\R{\mathbb{R}}

\def\Z{\mathbb{Z}}
\def\1{\mathbbold{1}}

\def\Vc{\mathcal{V}}

\def\Ts{\mathsf{T}}

\def\la{\lambda}

\def\d{\partial}
\def\0{\varnothing}
\def\sm{\setminus}

\def\le{\leqslant}
\def\ge{\geqslant}
\def\<{\langle}
\def\>{\rangle}

\def\Hom{\mathrm{Hom}}

\def\Len{\mathrm{Len}}
\def\area{\mathrm{area}}

\def\Gs{\mathsf{G}}
\def\Ts{\mathsf{T}}
\def\HG{\mathrm{HG}}
\def\Len{\mathrm{Len}}

\title{Polynomial valuations on plane polygons}

\author[A.~Khovanski\u{i}]{Askold Khovanski\u{i}}
\address[A.~Khovanski\u{i}]{Department of Mathematics, University of Toronto, 40 St. George St., Toronto, ON M5S 2E4 Canada}

\author[V.~Kiritchenko]{Valentina Kiritchenko}
\address[V.~Kiritchenko]{Institute for Information Transmission Problems\\
Bolshoy Karetny Pereulok 19, 127994 Moscow, Russia}

\author[V.~Timorin]{Vladlen Timorin}
\address[V.~Kiritchenko and V. Timorin]{Laboratory of Algebraic Geometry and Faculty of Mathematics\\
National Research University Higher School of Economics\\
Usacheva St. 6, 119048 Moscow, Russia}

\thanks{Research of the first named author was partially suported by the NSERC grant No. 156833-26.
The study was implemented in the framework of the Basic Research Program at HSE University (HSE-BR-2025-84).}

\begin{document}
\maketitle

\begin{abstract}
Scissors congruence problems involving translations have prompted the study of translation invariant simple valuations.
We review this classical theory from a naive and consistent viewpoint: starting from a description of
 all simple valuations on polygons, we characterize the effect of translation invariance.
A description of all polynomial simple valuations is obtained as a bi-product of the adopted approach
 and as a direct generalization of the translation invariant theory;
 it appears to be new.
\end{abstract}

\section{Introduction}
\label{s:intro}
Two plane polygons $P$ and $Q$ are \emph{scissors congruent} if one can cut $P$ into polygonal
 pieces and assemble $Q$ of these pieces.
Implicit in this definition is a certain group $\Gs$ of Euclidean isometries
 that allows moving the pieces around.
According to the choice of $\Gs$, one can talk of \emph{Euclidean scissors congruences} (where
 $\Gs$ consists of all Euclidean isometries), \emph{translation scissors congruences} (where $\Gs$
 coincides with the group of all translations), etc.
For a general choice of $\Gs$, one talks about \emph{$\Gs$-scissors congruence}.

Euclidean scissors congruence problems trace back to Antiquity.
According to the classical \emph{Wallace--Bolyai--Gerwien theorem}, two plane polygons
 are scissors congruent by means of Euclidean isometries if and only if they have equal areas.
In fact, the same result holds true when the group of all Euclidean isometries is replaced with the much smaller subgroup
 consisting only of parallel translations and half-turn rotations, as Hadwiger and Glur showed in \cite{HG51}.

Even in this simplest context, it is clear that certain abstract properties of the area play a decisive role,
 namely, additivity and $\Gs$-invariance.
Additivity property is formalized by the notion of a valuation.
Say that a real valued function $\mu$ on the set of all convex polygons in the plane $\R^2$ is a \emph{simple valuation} if
$$
\mu(P)=\mu(P_1)+\dots +\mu(P_n),
$$
 whenever $P$ is cut into finitely many polygonal pieces $P_1$, $\dots$, $P_n$.
More formally, $P$ is the union of the polygons $P_i$, and the interiors of these polygons are pairwise disjoint.
It suffices to require that $\mu(P)=\mu(P_+)+\mu(P_-)$ each time $P$ is
 dissected by a line into two parts $P_+$ and $P_-$.
Say that $\mu$ is \emph{$\Gs$-invariant} if $\mu(gP)=\mu(P)$ for all $g\in\Gs$.
For example, the Euclidean area is a valuation on (bounded) polygons
 that is invariant under the full group of Euclidean isometries.

It is clear from the definitions that, if $P$ and $Q$ are $\Gs$-scissors congruent,
 then $\mu(P)=\mu(Q)$, for every simple $\Gs$-invariant valuation $\mu$.
The converse is also true, as was proved by Hadwiger \cite{Had13}
 (we will also review the proof in this paper, at least for special choices of $\Gs$,
 see Section \ref{s:scissor}).
Solving a scissors congruence problem now reduces to finding a ``complete'' set
 of simple $\Gs$-invariant valuations that are enough to distinguish any pair
 of ($\Gs$-scissors) non-congruent polygons.
When $\Gs$ is the full group of Euclidean isometries, just one valuation,
 namely, the area, is enough.
Consider now translation scissors congruence problem in the plane.
For this problem, one needs an additional uncountably infinite family of valuations $\HG_L$ --- so called
 \emph{Hadwiger--Glur valuations} --- defined below.

Fix an orientation of $\R^2$ once and for all.
Also, it will be convenient to fix an origin $o$ of $\R^2$ and view $\R^2$ as a vector space.
Even though this is not strictly necessary, it allows to choose
 representatives in each class of parallel lines.
Choose an oriented line $L$ through the origin and define a function $\HG_L$ on convex polygons as follows.
Given a convex polygon $P$, let $L_\pm$ be the two support lines of $P$ parallel to $L$;
 we label these lines so that $P$ lies on the left from $L_+$ and on the right of $L_-$,
 with respect to the chosen orientation of $L$ as well as the chosen orientation of $\R^2$.
Define
$$
\HG_L(P):=\Len(P\cap L_+)-\Len(P\cap L_-),
$$
 where $\Len$ denotes the Euclidean length of a straight line interval.

A solution of the plane translation scissors congruence problem found by Hadwiger and Glur in \cite{HG51}
 goes as follows: two polygons $P$, $Q$ are translation scissors congruent if and only if
 they have the same area and the same values of all Hadwiger--Glur valuations.
Observe that, even though there are uncountably many Hadwiger--Glur valuations,
 it suffices to check only finitely many relations in order to decide whether $P$ and $Q$
 are scissors congruent, namely, that $\area(P)=\area(Q)$ and $\HG_L(P)=\HG_L(Q)$,
 for every line $L\ni o$ parallel to an edge of $P$ or to an edge of $Q$.
The proof of this result in \cite{HG51} is elementary; it is based on a direct geometric construction.
General facts of linear algebra allow to deduce a complete description of all
 translation invariant simple valuations on plane polygons.
Every such valuation is a (possibly, infinite) sum of $\xi\circ\area$ and $\xi_L\circ\HG_L$,
 where the functions $\xi$, $\xi_L:\R\to\R$ are linear over $\Q$ (not necessarily linear over $\R$)
 and the summation is performed over all lines $L$ through the origin.
In other words, valuations of the form $\xi\circ\area$ and $\xi_L\circ\HG_L$ \emph{generate}
 the space of simple translation invariant valuations by means of possibly infinite linear combinations;
 these infinite linear combinations make sense since, for every polygon $P$, only
 finitely many terms take nonzero values on $P$
 (indeed, only finitely many lines though the origin are parallel to the edges of $P$).
One can also describe all relations between $\xi\circ\area$ and $\xi_L\circ\HG_L$.

\begin{rem}
If a translation invariant valuation is nonnegative, i.e., if it takes nonnegative values on
 all polygons, then it is proportional to the area, with a real constant coefficient of proportionality.
Thus, the assumption of nonnegativity kills all terms of the form $\xi_L\circ\HG_L$,
 and the only surviving term of the form $\xi\circ\area$ is with $\xi$ being linear over $\R$,
 with nonnegative derivative.
\end{rem}

In this paper, we revisit the description of all translation invariant simple valuations on polygons.
Our approach is straightforward and natural; however, we are not aware of any previous implementation of this approach.
Namely, we start by describing \emph{all} simple valuations and then impose the condition of translation invariance.
All simple valuations $\mu$ on polygons can be described in the following two ways, cf. \cite{Dup01,KKT26}:
\begin{itemize}
  \item there is a \emph{cochain} $\theta$, i.e., a function on oriented straight line intervals
   that is additive under subdivisions, such that $\mu(P)$ is obtained by ``integrating'' $\theta$ over $\d P$,
   for every polygon $P$; a cochain $\theta$ represents zero valuation if and only if
   $\theta(a,b)=g(b)-g(a)$ for some function $g:\R^2\to\R$.
  \item there is a function $f$ on \emph{flags} $(a,L)$, where $L$ is a line in $\R^2$ and $a\in L$ is a point,
   such that $\mu(P)$, for a convex polygon $P$, is obtained as the sum of $\pm f(a,L)$ over all flags of $P$ (say that $(a,L)$
   is a \emph{flag of} $P$ if $a$ is a vertex of $P$ and $L$ contains an edge of $P$); a function $f$
   represents zero valuation if and only if $f(a,L)=g(a)+h(L)$ for some functions $g$, $h$.
\end{itemize}
\emph{Any function on flags} defines a simple valuation.
More details are given in Sections \ref{s:val-coc} and \ref{s:gen-rel}, see specifically Theorems \ref{t:val-coh} and \ref{t:flag-val}.
In terms of these descriptions, we study the meaning of translation invariance,
 and thus recover classical results of Hadwiger and Glur.
Finally (and this result appears to be new), we describe the condition that $\mu$ is a polynomial
 valuation via properties of an associated cochain $\theta$ and via properties of an associated function $f$ on flags
 (say that $\mu$ is polynomial of degree $\le d$ if, for every polygon $P$, the function $x\mapsto \mu(x+P)$
 is a degree $\le d$ polynomial function).
Theorem \ref{t:polyval-descr} describes degree $d$ polynomial simple valuations in terms of
 (1) degree $d+1$ polynomial differential 1-forms over $\Q$ and (2) degree $d$ polynomial cochains.

\begin{rem}[Nonsimple valuations]
All valuations considered in this paper are simple.
Nonsimple valuations are also worth studying 
 (e.g., the perimeter of $P$ and the sum of the values of a certain function over
 $P\cap\Z^2$ are nonsimple valuations of a polygon $P$).
Moreover, the space of all (not necessarily simple) valuations has a nice algebraic structure
 (its predual space is almost an algebra over $\R$, see \cite{Mc89}).
To a large extent, however, properties of nonsimple valuations reduce to those of simple ones.
Indeed, the quotient space of the space of all valuations modulo simple valuations
 embeds into the product $\prod_L \mathrm{Val}(L)$, with easily describable image,
 where $L$ ranges through all lines $L\subset\R^2$,
 and $\mathrm{Val}(L)$ stands for the space of all valuations on the intervals of $L$.
\end{rem}

\section{Simple valuations on a line}
\label{s:1D}
We start by overviewing a 1D version of the theory, which is straightforward.
A function $\mu$ on the compact intervals of the form $[a;b]\subset\R$ is called a
 \emph{simple valuation} (on intervals) if $\mu([a;a])=0$ for all $a\in\R$ and
$\mu([a;b])=\mu([a;c])+\mu([c;b])$
 whenever $c\in [a;b]$.
For example, the length of an interval $\Len([a;b]):=|b-a|$ is a simple valuation.
A wider class of examples is associated with an arbitrarily chosen function $f:\R\to\R$
 (no conditions whatsoever are imposed on $f$; e.g., it does not have to be continuous or measurable).
Every such $f$ defined the corresponding simple valuation $\mu_f$ defined
 by the condition $\mu_f([a;b])=f(b)-f(a)$ whenever $a\le b$.
Note that the valuation $\Len$ also has the form $\mu_f$, for $f(x)=x$.

\begin{thm}
\label{t:all-1D}
Any simple valuation on intervals has the form $\mu_f$ for some function $f:\R\to\R$.
Under the additional assumption that $f(0)=0$, a function $f$ with $\mu=\mu_f$ is unique.
In general, $\mu_f=0$ if and only if $f$ is constant.
\end{thm}

\begin{proof}
Let $\mu$ be a simple valuation on intervals.
Set $f(x)=\mu([0;x])$ if $x\ge 0$ and $f(x)=-\mu([x;0])$ otherwise.
Now, if $a\le 0\le b$, then
$$
\mu([a;b])=\mu([a;0])+\mu([0;b])=-f(a)+f(b)=\mu_f([a;b]),
$$
 by definition of $f$.
Other possible locations of $[a;b]$ relative to the point $0$
 yield the same result, through similar computations;
 these computations are left to the reader.
We see that $\mu$ has the form $\mu_f$.

If $\mu=\mu_f$ and $f(0)=0$, then $f$ can be uniquely recovered from $\mu$:
 indeed, $f(x)=\mu_f([0;x])$ for $x\ge 0$ by definition of $\mu_f$,
 and $f(x)=-\mu_f([x;0])$ for $x\le 0$.
Assume now that $\mu_f=0$ without imposing the normalization $f(0)=0$.
Then the function $g(x):=f(x)-f(0)$ satisfies this normalization together with
 the condition $\mu_g=0$.
Hence, $g=0$, which means that $g$ is the constant function $f(0)$.
Alternatively, $\mu_f=0$ means $f(b)=f(a)$ whenever $a\le b$,
 which also immediately implies that $f$ is constant.
\end{proof}

Let $\Ts$ be the group of all translations of $\R$.
One can identify $\Ts$ with $\R$.
Under this identification, $v\in\Ts$ acts as the parallel translation $x\mapsto x+v$,
 in particular, it sends an interval $[a;b]$ to $[a+v,b+v]$.
A valuation $\mu$ on intervals is said to be \emph{translation invariant} if
 it is stable under the action of $\Ts$, that is, if $\mu([a+v,b+v])=\mu([a;b])$,
 for all $a$, $b$, $v\in\R$.
Below is a characterization of all translation invariant simple valuations on intervals.

\begin{thm}
\label{t:Tinv-1D}
Every translation invariant simple valuation on intervals has the form $\mu_f$,
 for a unique $\Q$-linear function $f$.
The same valuation can also be represented as $f\circ \Len$.
\end{thm}

\begin{proof}
Consider a translation invariant simple valuation $\mu$ on intervals.
By Theorem \ref{t:all-1D}, there is a unique function $f:\R\to\R$ such that $f(0)=0$ and $\mu_f=\mu$.
We claim that $f$ is \emph{additive}: it satisfies the identity $f(x+y)=f(x)+f(y)$ for all $x$, $y\in\R$.
It is straightforward and well known that additive functions are the same as $\Q$-linear functions.
Additivity of $f$ follows from additivity and translation invariance of $\mu$:
$$
f(x+y)=\mu([0,x+y])=\mu([0;x])+\mu([x;x+y])=f(x)+f(y).
$$
Here, the first equality is the definition of $\mu_f=\mu$, the second equality is
 additivity $\mu$ (a consequence of $\mu$ being a simple valuation),
 and the last equality is again the definition of $\mu_f$ combined with
 translation invariance (namely, $\mu([x;x+y])=\mu([0;y])$).

The last claim of the theorem is immediate from the observation that
 $\Len([0;x])=x$ and $\mu([0;x])=f(x)$ for $x\ge 0$.
\end{proof}

We see that any translation invariant simple valuation on intervals can be
 expressed as an additive function of the length.
This is closely related to the obvious solution of the 1D translation scissors congruence problem:
 \emph{two intervals in $\R$ are $\Ts$-scissors congruent if and only if they have the same length}.
A more general class than translation invariant valuations are polynomial valuations.
In order to define them, we first recall the notion of a polynomial function
 (all our polynomial functions are over $\Q$, not over $\R$).

\begin{dfn}[Polynomial functions]
\label{d:poly}
Let $A$ be an abelian group, and consider a function $f:A\to\R$.
Say that $f$ is a \emph{polynomial function of degree $\le 0$} if it is constant.
By induction on $d\in\Z_{>0}$, define a \emph{polynomial function of degree $\le d$} as
  a function $f:A\to\R$ such that, for each $a\in A$, the function $\Delta_af(x):=f(x+a)-f(x)$
  is a polynomial  function of degree $\le d-1$.
General properties of polynomial functions on abelian groups, and even abelian semigroups,
 are discussed in \cite{Kho25}.
\end{dfn}

We will always assume that $A$ is a vector space over $\Q$.
Any polynomial function can be represented as a sum of homogeneous (over $\Q$)
 components, as is well known.
It will be useful in the sequel that two difference operators commute:
$$
\Delta_a\Delta_b f=\Delta_b\Delta_a f,
$$
 for any function $f:A\to\R$, as a straightforward computation yields.
For example, if $f$ is a polynomial of degree $\le 1$, then $\Delta_af(x)$ is a constant, for each $a\in\R$.
Assume additionally that $f(0)=0$.
Then
$$
f(x+a)-f(x)=\Delta_a f(x)=\Delta_a f(0)=f(a),\quad \forall x,a\in\R,
$$
 which means that $f$ is \emph{additive}, i.e, $\Q$-linear.
Conversely, $f$ being additive implies that $f+c$ is a polynomial of degree $\le 1$,
 for every $c\in\R$.
In particular, polynomial functions $f:\R\to\R$ do not have to be continuous,
 according to our chosen definition.

\begin{dfn}[Polynomial valuations]
Let $\mu$ be a simple valuation on intervals.
Say that $\mu$ is \emph{polynomial of degree $\le d$} if, for every interval $I$,
 the function $v\mapsto \mu(I+v)$ is a polynomial function of degree $\le d$.
\end{dfn}

Clearly, the same definition also applies to simple valuations on plane polygons;
 we will consider polynomial valuations on polygons later on.

\begin{thm}
\label{t:poly-1D}
Any degree $\le d$ polynomial simple valuation on intervals has the form $\mu_f$,
 for some polynomial function $f$ of degree $\le d+1$.
Conversely, if $f$ has degree $\le d+1$, then $\mu_f$ is a polynomial degree $\le d$ valuation.
\end{thm}

\begin{proof}
Start with the last claim of the theorem.
Let $f$ be a polynomial function of degree $\le d+1$, and consider $\mu_f$.
For every interval $I=[a;b]$ with $a\le b$, one has
$$
\mu_f(I+v)=f(b+v)-f(a+v)=\Delta_{b}f(v)-\Delta_a f(v),
$$
 which is by definition a degree $\le d$ polynomial function of $v$.
Hence, $\mu_f$ is a degree $\le d$ polynomial valuation, as claimed.

Now consider a degree $\le d$ polynomial valuation $\mu$.
By Theorem \ref{t:all-1D}, it has the form $\mu=\mu_f$ for some function $f:\R\to\R$
 with the property $f(0)=0$.
One has
$$
\Delta_a f(x)=f(x+a)-f(x)=\mu([x;x+a])=\mu([0;a]+x),
$$
 for all $x\in\R$ and $a\in\R_{\ge 0}$.
Fixing $a\ge 0$, we see that $\Delta_a f(x)$ is a degree $\le d$ polynomial function of $x$.
Similar computations yield the same conclusion for $a<0$.
It follows that $f$ is a polynomial function of degree $\le d+1$.
\end{proof}

Theorem \ref{t:poly-1D} can be restated as a canonical isomorphism
 between degree $\le d$ polynomial simple valuations on intervals
 and degree $\le d+1$ polynomial functions, modulo constants.

\section{Valuations and cochains}
\label{s:val-coc}
Even though \emph{cochains} appear in every cohomology theory, we will use the name
 only in one specific context connected with valuation theory.
The same applies to the term \emph{coboundary}.

\begin{dfn}[Cochains and coboundaries]
A \emph{cochain} is a function $\theta:\R^2\times\R^2\to\R$
 with the property that $\theta(a,b)+\theta(b,c)=\theta(a,c)$ for all triples of \textbf{collinear} points $a$, $b$, $c\in\R^2$.
Given a function $g:\R^2\to\R$, define the cochain $dg$ by the formula $dg(a,b)=g(b)-g(a)$;
 call $dg$ a \emph{coboundary}.
It is straightforward that $dg$ is indeed a cochain.
Thus, the space of coboundaries is a vector subspace in the vector space of all cochains.
\end{dfn}

One can think of a cochain $\theta$ as a function $[a;b]\to \theta(a,b)$ on oriented intervals $[a;b]$
 that is additive under subdivision.
Some basic properties of cochains are immediate from the definition.
For example, $\theta(a,a)+\theta(a,a)=\theta(a,a)$ implies that $\theta(a,a)=0$ for every $a\in\R^2$.
Next, the \emph{cochain condition}
$$
\theta(a,b)+\theta(b,a)=\theta(a,a)=0
$$
 implies that $\theta$ is skew symmetric in its two arguments.
In other words, $\theta$ can be viewed as a function on oriented straight line segments
 that vanishes on degenerate (=dimension 0) segments and that changes sign
 whenever the orientation of a segment is reversed.

One can integrate cochains over 1-chains.
A \emph{1-chain} in $\R^2$ can be defined as an integer linear combination of
 oriented straight line segments, i.e., as a formal sum of the form
$$
c=m_1 [a_1;b_1] +\dots + m_k[a_k;b_k],
$$
 where $a_i\ne b_i$ and $m_i\in\Z$, for all $i=1$, $\dots$, $k$.
Here, $[a;b]$ is the straight line segment connecting the points $a$, $b$
 and oriented from $a$ to $b$.
Clearly, 1-chains form an abelian group under addition.

\begin{dfn}[Integrals]
Consider a cycle $\theta$ and a 1-chain $c$ as above.
Set
$$
\int_c \theta := m_1\theta(a_1,b_1)+\dots +m_k\theta(a_k,b_k).
$$
 and call this number the \emph{integral of $\theta$ over $c$}.
\end{dfn}

It follows from the cochain condition that $\int_c\theta=0$ for
 every cochain $\theta$ and every 1-chain $c$ from the subgroup
 generated by 1-chains of the form $[a;b]-[a;c]-[c;b]$ for $c\in [a;b]$.
Call this subgroup the \emph{subgroup of null chains}.
Vanishing of the integral of $\theta$ over all null chains can also be restated
 as the invariance of the integral under subdivisions.
Since null chains do not change the integral, they will be often ignored.

For every convex polygon $P$ in the plane, let $\d P$ be the 1-chain
 obtained  as the sum of all edges of $P$ oriented so that $P$ is on their left
 (more formally: the orientation of $\d P$ matches that of the boundary of $P$
 as an oriented manifold with the boundary).
With every cochain $\theta$, one associates the valuation $S\theta$ as follows:
 for a convex polygon $P$, one has
$$
S\theta(P)=\int_{\d P} \theta.
$$
It is indeed straightforward to verify that $S\theta$ is a simple valuation.

Fix two distinct lines $L_0$, $L_1$ through the origin.
Visualize them as the horizontal and the vertical coordinate axes, respectively.
A cochain $\theta$ is said to be \emph{normalized} if $\theta(a,b)=0$ whenever
 $b\in a+L_1$ and whenever $a$, $b\in L_0$.
In other words, a normalized $\theta$ vanishes on every vertical line
 as well as on the horizontal axis $L_0$.
Known descriptions of simple valuations
 are equivalent in the 2D case to the next theorem, even though they are usually formulated differently.

\begin{thm}
  \label{t:val-coc}
Every simple valuation on plane polygons has the form $S\theta$ for a
 unique normalized cochain $\theta$.
\end{thm}

\begin{proof}
Let $\mu$ be a simple valuation on plane polygons; one needs to find a normalized cochain $\theta$ such that $\mu=S\theta$.
Consider an oriented segment $I=[a;b]$; we want to define $\theta(a,b)$.
Assuming that $I$ is not vertical and that $I$ is disjoint from $L_0$ except possibly at an endpoint,
 form a trapezoid $T_{a,b}$ with two vertical edges, one edge in $L_0$, and the remaining edge coinciding with $I$.
Such a trapezoid $T_{a,b}$ is uniquely defined by $a$ and $b$; it will be called a \emph{coordinate trapezoid};
 see Fig. \ref{fig:tra1}, left.
Define $\theta(a,b)$ as $\pm \mu(T_{a,b})$, where the sign is plus or minus
 depending on whether the orientation of $I$ (from $a$ to $b$) corresponds to the orientation of $\d T_{a,b}$.
If $I$ intersects $L_0$ at a point $c$, then set $\theta(a,b):=\theta(a,c)+\theta(c,b)$, where
 the terms in the right-hand side have already been defined.
Finally, set $\theta(a,b)=0$ if either $I$ is vertical or $I\subset L_0$.
By now, we have defined a cochain $\theta$ using the values of $\mu$ on coordinate trapezoids.
It is straightforward to verify that $\theta$ is indeed a cochain; this follows from the additivity of $\mu$.

\begin{figure}
  \centering
  \includegraphics[width=.8\textwidth]{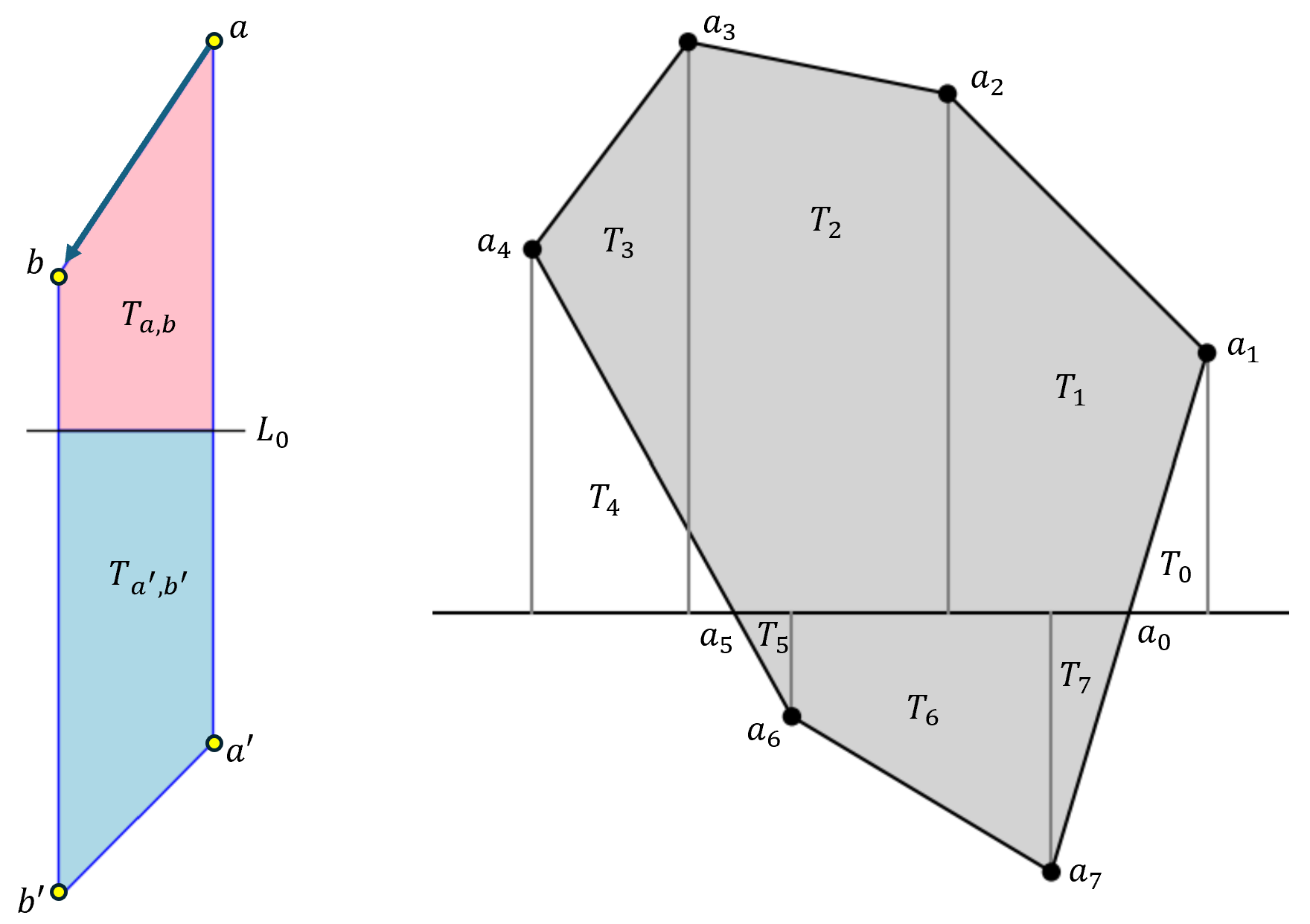}
  \caption{\small Left: a coordinate trapezoid $T_{a,b}$.
  Note that $\theta(a,b)$ is by definition $\mu(T_{a,b})$.
  On the other hand, $\theta(a',b')=-\mu(T_{a',b'})$.
Right: a convex polygon $P$ with vertices $a_0$, $\dots$, $a_7$,
 where $a_0$ and $a_5$ are ``artificial'' vertices marked at the intersections of $\d P$ with $L_0$.
 For every simple valuation $\mu$, the value $\mu(P)$ can be computed as the sum of
 $\pm\mu(T_i)$, where $i=0$, $\dots$, $7$, and the sign is minus for $T_0$ and $T_4$
 (it is plus for all $T_i$ with $i\ne 0,4$).
}
\label{fig:tra1}
\end{figure}

Now consider a convex polygon $P$ with vertices $a_0$, $\dots$, $a_{n-1}$, listed in the counterclockwise order.
Somewhat artificially, we add (at most two) points, where the boundary of $P$ intersects $L_0$, to the list of vertices.
Then
$$
S\theta(P)=\int_{\d P}\theta=\sum_{i=0}^{n-1}\theta(a_i,a_{i+1}).
$$
Here, the index $i$ is viewed as an element of $\Z/n\Z$;
 the first equality is the definition of $\d P$, and the second equality is the definition of the integral.
By the above $\theta(a_i,a_{i+1})$ equals $\pm\mu(T_i)$, where $T_i=T_{a_i,a_{i+1}}$.
On the other hand, $P$ is represented as a set-theoretic union or difference of the
 coordinate trapezoids $T_i$ so that the sign in the equality $\theta(a_i,a_{i+1})=\pm\mu(T_i)$
 is plus or minus according to whether one needs to add or subtract $T_i$ in
 the just described representation of $P$.
It follows from the additivity of $\mu$ that $S\theta(P)=\mu(P)$.

What remains to prove is that $S\theta=0$ for a normalized $\theta$ implies $\theta=0$.
Indeed, due to normalization conditions, one has to have
$$
S\theta(T_{a,b})=\pm\theta(a,b)
$$
 for every coordinate trapezoid $T_{a,b}$, where the sign is plus or minus according to
 whether the direction from $a$ to $b$ matches the orientation of the boundary of $T_{a,b}$.
By the above formula, $S\theta=0$ implies $\theta(a,b)=0$ at least for all $[a;b]$
 that do not cross $L_0$; it follows that $\theta(a,b)=0$ for all $a$, $b\in \R^2$ by additivity.
\end{proof}

For example, the area valuation is obtained from a certain cochain.

\begin{ex}
\label{ex:area-tha}
Define a cochain $\theta^{\area}$ as
$$
\theta^{\area}(a,b):=\int_{a}^{b} x_0dx_1,
$$
 where $(x_0,x_1)$ is a Euclidean coordinate system with coordinate axes $L_0$, $L_1$.
Then, as is easily seen, the area valuation coincides with $S\theta^{\area}$.
\end{ex}

On the one hand, Theorem \ref{t:val-coc} associates a unique cochain with each simple valuation.
On the other hand, the construction involves arbitrary choices, namely,
 the two lines $L_0$ and $L_1$ can be chosen arbitrarily.
A more invariant description of simple valuations goes as follows.

\begin{thm}
  \label{t:val-coh}
The map $\theta\mapsto S\theta$ from the space of cochains to the space of simple valuations
 is an epimorphism of $\R$-vector spaces whose kernel consists precisely of all coboundaries.
\end{thm}

\begin{proof}
Clearly, $S(dg)=0$, for every function $g:\R^2\to\R$.
Suppose now that $S\theta=0$ for some cochain $\theta$, and we want to prove that $\theta$ is a coboundary.
Fix a linear coordinate system $(x_0,x_1)$ on $\R^2$, and let $L_0$, $L_1$ be the corresponding
 coordinate axes given by $\{x_1=0\}$, $\{x_0=1\}$, respectively.
Define the function $g:\R^2\to\R$ by the formula
$$
g(x_0,x_1)=\theta((0,0),(0,x_1)) + \theta((0,x_1),(x_0,x_1)).
$$
In other words, the value of $g$ at the point $(x_0,x_1)$ is by definition the integral of $\theta$ over
 the $\Gamma$-shaped broken line connecting the origin with $(x_0,x_1)$
 via a vertical segment and then a horizontal segment.
We claim that the cochain $\theta-dg$ is normalized, as follows immediately from the definition of $g$.
By Theorem \ref{t:val-coc}, the cochain $\theta-dg$ being normalized and satisfying $S(\theta-dg)=0$
 implies that $\theta-dg=0$.
Hence $\theta$ is a coboundary, as claimed.
\end{proof}

Choose some line $L$ in the plane.
Define the \emph{group of $L$-chains} as the subgroup of the group of 1-chains
 generated by oriented straight line segments parallel to $L$.
Given a cochain $\theta$, let $\theta^L$ be the cochain whose values
 on all $L$-chains coincides with the values of $\theta$ and that
 is zero on any straight line segment not parallel to $L$.
Every cochain $\theta$ can be written as
$$
\theta=\sum_L \theta^L,
$$
where the sum is over all lines through the origin.
Note that the infinite sum in the right hand side makes sense since,
 for every polygon $P$, only finitely many terms of this sum take nonzero value on $P$.
Defining a cochain $\theta$ is equivalent to defining a family of cochains $\theta^L$
 labeled by lines $L$ through the origin.

\section{Generators and relations}
\label{s:gen-rel}
We can now describe the space of simple valuations on plane polygons in terms of generators and relations.
Generators correspond to \emph{flags}, i.e., pairs of the form $(a,L)$,
 where $L$ is an affine line in $\R^2$, and $a\in L$ is a point.
The \emph{flag valuation} $\Upsilon_{a,L}$ takes value 0 on a convex polygon $P$
 unless $a$ is a vertex of $P$ and $L$ contains an edge of $P$.
In the latter case, $\Upsilon_{a,L}(P)=\pm 1$ depending on whether the rotation around $a$
 from $P\cap L$ towards the interior of $P$ is positive or negative with respect to the chosen orientation of $\R^2$,
  see Fig. \ref{fig:Ups}.
Given any function $f$ on flags, consider the sum
$$
\mu_f=\sum_{(a,L)} f(a,L)\Upsilon_{a,L},
$$
over all flags $(a,L)$.
This formally infinite sum makes sense since only finitely many terms may attain
 nonzero values on any given polygon.
Below is a description of simple valuations given in \cite{KKT26}; it can
 also be deduced from \cite{Dup01}.

\begin{figure}
  \centering
  \includegraphics[height=5cm]{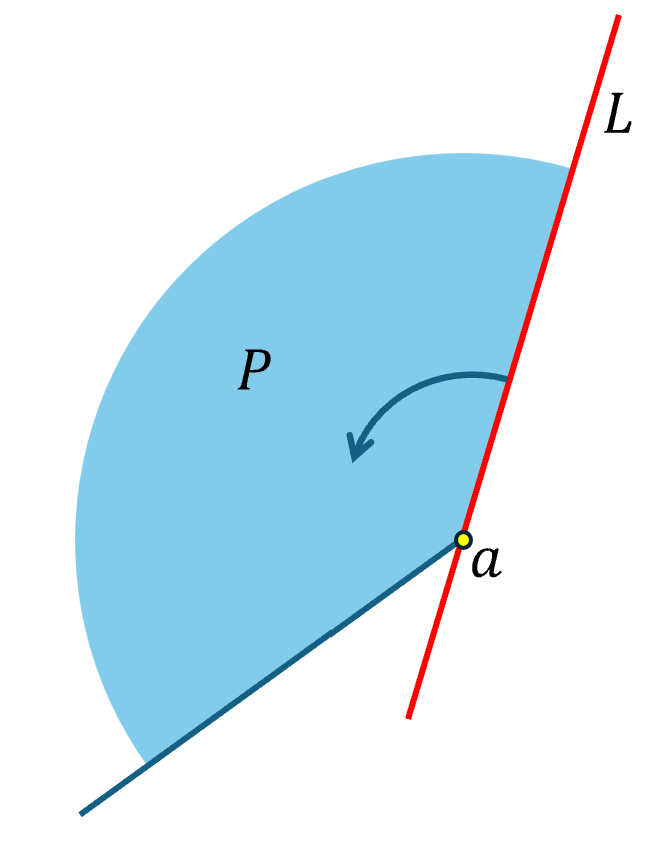}
  \caption{\small The shaded region represents a part of a convex polygon $P$ near its vertex $a$.
  The line $L$ is a support line of $P$ containing $a$. In the figure, $\Upsilon_{a,L}(P)=+1$
  since, as $L$ rotates around $a$ towards $P$ (as the arrow shows), it rotates in the positive direction.}
  \label{fig:Ups}
\end{figure}

\begin{thm}
  \label{t:flag-val}
Any simple valuation on plane polygons has the form $\mu_f$, for some function $f$ on flags.
Moreover, $\mu_f=0$ if and only if $f$ can be represented as $f(a,L)=g(a) + h(L)$,
 for some functions $g$ and $h$, so that $g(a)$ depends only on $a$ and $h(L)$ depends only on $L$.
\end{thm}

Thus, flag valuations generate the space of all simple valuations in the sense
 that any simple valuation on plane polygons is a (possibly, infinite) linear
 combination of flag valuations.
Every relation between flag valuation is a (possibly, infinite) linear combination of
 the following two types of relations:
$$
\sum_{L\ni a} \Upsilon_{a,L}=0,\quad \sum_{a\in L} \Upsilon_{a,L}=0.
$$
In the first relation, $a$ is fixed and $L$ ranges through all lines through $a$,
 while, in the second relation, $L$ is fixed, and $a$ ranges through all points of $L$.

\begin{proof}[Proof of Theorem \ref{t:flag-val}]
Let $\mu$ be any simple valuation on polygons.
Fix two distinct lines $L_0$, $L_1$ through the origin and visualize them
 as the horizontal and vertical coordinate lines.
By Theorem \ref{t:val-coc}, one can represent $\mu$ as $S\theta$, for
  a unique normalized cochain $\theta$.
Given $\theta$, we can now define a function $f$ on flags as follows.
If, for a flag $(a,L)$, the line $L$ is vertical, then set $f(a,L)=0$.
Otherwise, let $b$ be the intersection point of $L$ with $L_1$.
Set $f(a,L)=\theta(a,b)$.
We have thus defined the function $f$ on all flags, and we claim that $\mu_f=\mu$.

Given a polygon $P$, evaluating $\mu_f(P)$ involves a summation over all edges of $P$.
For an edge $e$ of $P$ with endpoints $a_\pm$ that lies in a line $L$,
 the contribution of $e$ to $\mu_f(P)$ consists of two terms
$$
f(a_-,L)-f(a_+,L)=\theta(a_-,a_+),
$$
 where the endpoints $a_\pm$ are labeled so that the direction from $a_-$ to $a_+$
 corresponds to the orientation of $\d P$.
As is immediate from the above formula, the computation of $\mu_f(P)$
 now reduces to integrating $\theta$ over $\d P$.
We have thus proved that $\mu=\mu_f$.

Suppose now that $\mu_f=0$, and define a cochain $\theta$ by the formula
 $\theta(a,b):=f(a,L)-f(b,L)$, where $a\ne b$ and $L$ is the line through $a$ and $b$.
By the above, $\mu_f(P)=\int_{\d P}\theta$, for every convex polygon $P$.
In particular, the cohomology class of $\theta$ is zero, that is, $\theta=dg$ for some function $g:\R^2\to\R$.
The function $g$ can also be viewed as a function on the flags $(a,L)$ that depends only on $a$, not on $L$.
Replacing $f$ with $f+g$, we can arrange that $\theta=0$.
However, the latter means that $f(a,L)=f(b,L)$ for any pair of distinct points $a$, $b$ on a line $L$,
 i.e., the value $f(a,L)$ depends only on $L$, not on $a\in L$.
This proves the last claim of the theorem.
\end{proof}

We can also describe a basis in the space of simple valuations on polygons, i.e.,
 a collection of simple valuations with the property that any simple valuation can
 be uniquely represented as a (possibly, infinite) linear combination of the chosen collection.

\begin{thm}
\label{t:flag-norm}
Consider the flag valuations $\Upsilon_{a,L}$ associated with the flags $(a,L)$ such that:
$$
(1)\ L \hbox{ is not parallel to } L_1;\quad (2)\ L\ne L_0;\quad (3)\ a\notin L_1.
$$
These flag valuations form a basis in the space of all simple valuations on polygons.
\end{thm}

\begin{proof}
Say that a function $f$ on flags is \emph{normalized} if $f(a,L)=0$ for all flags that
 do \emph{not} satisfy (1) -- (3), that is, if $L$ is vertical (i.e., parallel to $L_1$), or if $L=L_0$, or if $a\in L_1$.
Restating the theorem, we need to show that every simple valuation $\mu$ has
 the form $\mu_f$ for a unique normalized $f$.
Existence follows from Theorem \ref{t:flag-val}: in fact, the function $f$ defined
 in the proof of this theorem is normalized.
Assume now that $f$ is normalized and that $\mu_f=0$; we want to show that $f=0$.
Note first that $\mu_f=S\theta$, where $\theta(a,b)=f(a,L)-f(b,L)$, and $L$ is the line through $a$ and $b$;
 see again the proof of Theorem \ref{t:flag-val}.
Since $\theta$ is normalized and $S\theta=0$, one has $\theta=0$, by Theorem \ref{t:val-coc}.
It follows that $f(a,L)$ depends only on $L$, not on $a$, that is, $f(a,L)=h(L)$
 for some function $h$ on lines.
On the other hand, $f(a,L)=0$ for $a\in L_1$ implies that $h(L)=0$ for any non-vertical line.
For vertical lines $L$, we also have $h(L)=0$, by the normalization assumption.
Thus, $f=0$, as desired.
\end{proof}

\section{Associated 2-forms}
\label{s:Tinv-2D}
A valuation $\mu$ on polygons in $\R^2$ is said to be \emph{polynomial of degree $\le d$}
 if $\mu(P+v)$ is a polynomial function of $v\in\R^2$ of degree $\le d$ (see Definition \ref{d:poly}), for every convex polygon $P$.
In particular, if $\mu$ is polynomial of degree $\le 0$, then we say that $\mu$ is \emph{translation invariant}.
Below, we associate a certain skew symmetric $\Q$-bilinear function with every translation invariant simple valuation.

\begin{dfn}[Associated 2-form]
\label{d:assform}
With every translation invariant simple valuation $\mu$ on polygons, associate
 a function $\omega_\mu:\R^2\times\R^2\to \R$ as follows.
Let $\theta$ be a cochain with $S\theta=\mu$ and, for each pair of vectors $u$, $v\in\R^2$
 write $\Pi(u,v)$ for the parallelogram spanned by $u$ and $v$, that is,
 the parallelogram with vertices $o$, $u$, $u+v$, $v$ (see Fig. \ref{fig:omega}, left);
 the indicated order of the vertices defines an orientation of $\Pi(u,v)$.
Set $\omega_\mu(u,v)$ to be the integral of $\theta$ over the boundary of $\Pi(u,v)$,
 where the orientation of the boundary corresponds to the chosen orientation of $\Pi(u,v)$.
Observe that $\Pi(u,v)$ may degenerate into an interval or a point, in which case we set $\omega_\mu(u,v)=0$.
Call $\omega_\mu$ the \emph{associated 2-form} of $\mu$.
\end{dfn}

For example, if $\mu$ is the area, then $\omega_\mu$ is the associated area form.
It is easy to that that $\omega_\mu(u,v)$ is well defined, i.e., independent
 of the choice of $\theta$ in its cohomology class: indeed, the integral
 of a coboundary over $\d\Pi(u,v)$ is zero.
Also, note that $\omega_\mu(u,v)=\pm \mu(\Pi(u,v))$, where
 the sign depends on the order of $u$ and $v$:
 it is plus if $u$, $v$ define the positive orientation of the plane, and it is minus otherwise.
Observe that $\omega_\mu$ is skew symmetric, i.e., $\omega_\mu(v,u)=-\omega_\mu(u,v)$,
 for the reason just mentioned.

\begin{lem}
\label{l:om-bilin}
The form $\omega_\mu$ is bilinear over $\Q$.
\end{lem}

\begin{proof}
Due to skew symmetry, it suffices to establish that
$$
\omega_\mu(u_1+u_2,v)=\omega_\mu(u_1,v)+\omega_\mu(u_2,v),
$$
 for all triples of vectors $u_1$, $u_2$, $v\in\R^2$.
Set $\Pi=\Pi(u,v)$ and $\Pi_i(u_i,v)$ for $i=1$, $2$.
Also, denote by $T$ the triangle with vertices $o$, $u$ and $u_1$.
Note that (see Fig. \ref{fig:omega})
$$
\d\Pi =\d\Pi_1+\d(\Pi_2+u_1)+\d T-\d (T+v)
$$
Fix a cochain $\theta$ such that $\mu=S\theta$ and integrate $\theta$ over $\d\Pi$.
On the one hand, the integral equals $\omega_\mu(u,v)$, by Definition \ref{d:assform}.
On the other hand, using the formula displayed above,
 the integral reduces to
$$
\int_{\d\Pi_1}\theta + \int_{\d(\Pi_2+u_1)}\theta=\omega_\mu(u_1,v)+\omega_\mu(u_2,v).
$$
Here, we used that $\mu$ is translation invariant, so that the integral of $\theta$
 over $\d(\Pi_2+u_1)$ equals the integral over $\Pi_2$, and
 the integral over $\d T$ equals that over $\d (T+v)$.
\end{proof}

\begin{figure}
  \centering
  \includegraphics[height=5cm]{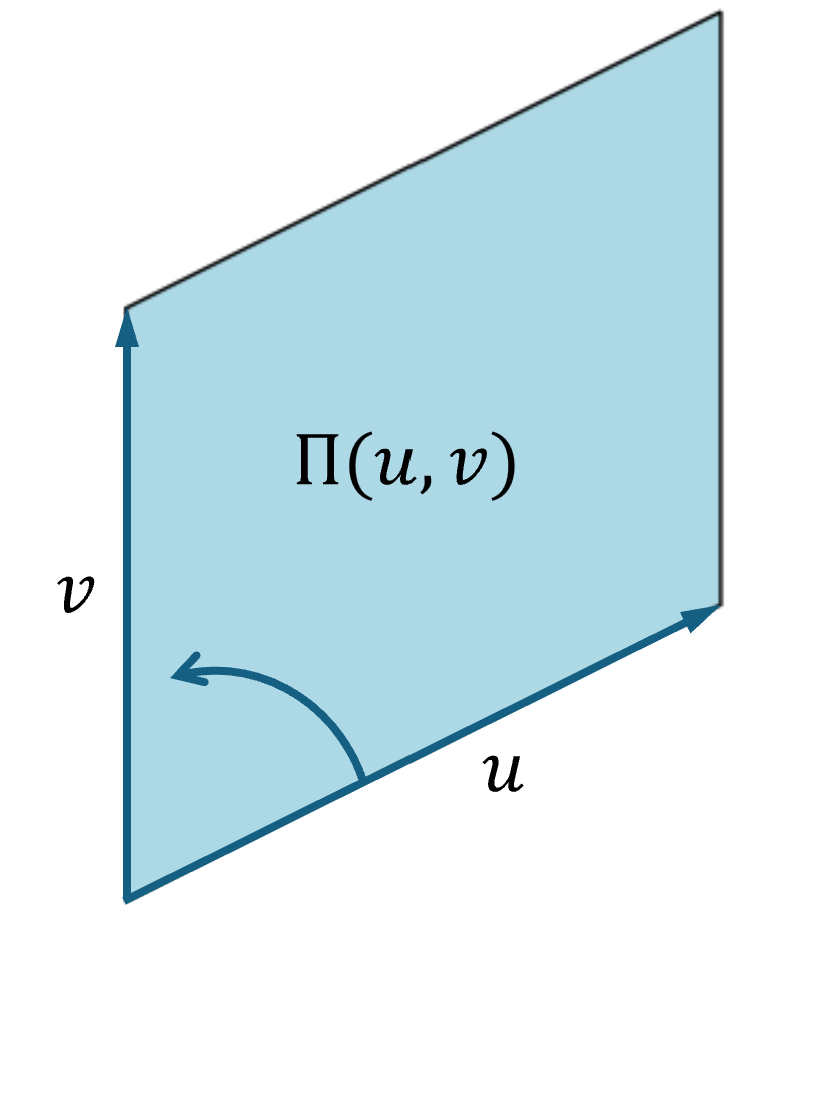}\hspace{1cm}
  \includegraphics[height=5cm]{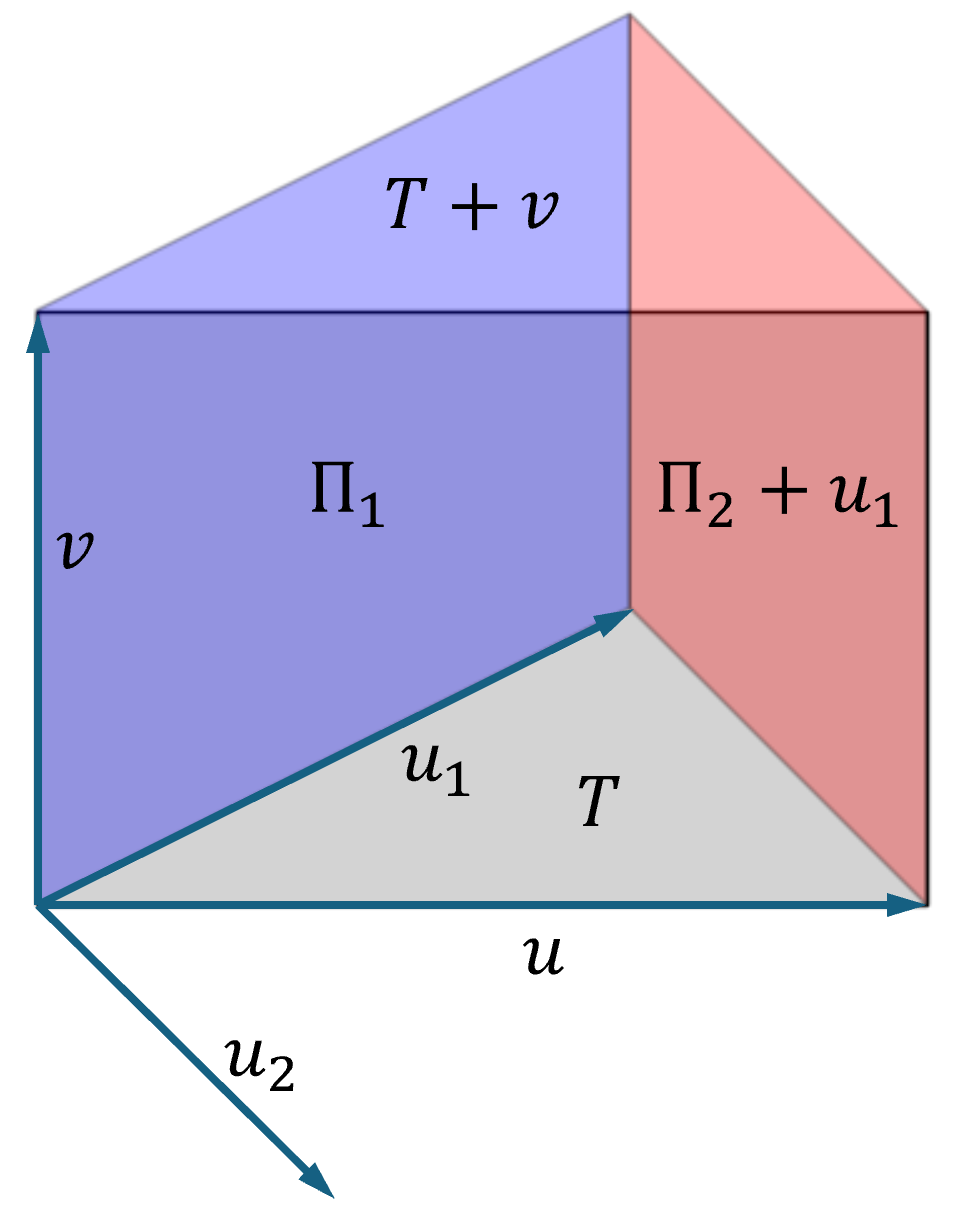}
  \caption{\small Left: the parallelogram $\Pi(u,v)$ spanned by vectors $u$ and $v$.
  Right: additivity of $\omega_\mu$ as a function of its first argument:
  $\omega_\mu(u_1+u_2,v)=\omega_\mu(u_1,v)+\omega_\mu(u_2,v)$.}\label{fig:omega}
\end{figure}

Note that the claim of Lemma \ref{l:om-bilin} can also be deduced directly
 from the additivity of $\mu$.
However, this requires case study.
For example, the cases $T\subset \Pi$ and $T\not\subset\Pi$ correspond to
 different configurations of polygons, cf. Fig. \ref{fig:omega}, right.
Dealing only with the boundaries of these polygons simplifies the consideration.

Another useful property of $\omega_\mu$ follows from $\mu$ being simple
 (the proof is straightforward and is left to the reader):

\begin{lem}
\label{l:collin}
Let $\omega_\mu$ be the associated 2-form of a translation invariant simple valuation $\mu$.
If $u$ and $v$ are linearly dependent over $\R$, then $\omega_\mu(u,v)=0$.
\end{lem}

By a (constant) \emph{2-form over $\Q$}, we will mean a $\Q$-bilinear skew symmetric function
 $\omega:\R^2\times\R^2\to\R$.
Say that $\omega$ is \emph{simple}
 if it satisfies the conclusion of Lemma \ref{l:collin}, that is, if $\omega(u,v)=0$
 for every pair of collinear vectors $u$, $v$.
There is a direct characterization of simple 2-forms over $\Q$, given in the next proposition.

\begin{prop}
\label{p:simp2form}
Let $\omega$ be a simple 2-form over $\Q$.
Then there is a $\Q$-linear map $\xi:\bigwedge^2_\R\R^2\to \Q$ with the property
 $\omega(u,v)=\xi(u\wedge_\R v)$, for all $u$, $v\in\R^2$.
More geometrically, $\omega(u,v)$ depends only on the area of the parallelogram
 spanned by $u$ and $v$, and this dependence is $\Q$-linear.
\end{prop}

\begin{proof}
Choose a basis of $\R^2$ as a real vector space.
With this chosen basis, vectors in $\R^2$ are identified with columns of size 2.
Pairs of vectors are then represented by $2\times 2$ matrices;
 we may write $(u,v)$ both for a pair of vectors $u$, $v$ and for the corresponding $2\times 2$ matrix.
The following \emph{column operations on matrices} do not change the value of $\omega$:
\begin{enumerate}
  \item $(u,v)\mapsto (u,v+\lambda u)$, where $\lambda\in\R$;
  \item $(u,v)\mapsto (u+\lambda v,v)$, where $\lambda\in\R$;
\end{enumerate}

The operations preserve not only $\omega(u,v)$ but also $\det(u,v)$.
Suppose that $\det(u,v)\ne 0$.
We claim that, using operations (1) -- (2), one can reduce $(u,v)$ to the diagonal form $\mathrm{diag}(1,\delta)$,
 where $\delta=\det(u,v)$, as follows.
Look at the first row.
Both entries of the first row cannot be zero simultaneously, by our assumption $\det(u,v)\ne 0$.
Using operations (1) and (2), we can then arrange that both entries are nonzero.
Operation (2) allows to make the $(1,1)$-entry equal to one.
Then operation (1) kills the $(1,2)$-entry, so that the first row is now $(1,0)$.
Proceed with the second row.
Note that the $(2,2)$-entry is now $\delta\ne 0$.
Operation (2) kills the $(2,1)$-entry, and we are done.

Evaluating $\omega$ on the diagonal matrix  $\mathrm{diag}(1,\delta)$ defines
 a certain function $\xi(\delta)$ of $\delta$, which is clearly $\Q$-linear.
We have
$$
\omega(u,v)=\omega(\mathrm{diag}(1,\delta))=\xi(\delta)=\xi(\det(u,v)),
$$
 which completes the proof.
\end{proof}

\section{Hadwiger--Glur classification}
\label{ss:HG}
Recall that $L_0$, $L_1$ are two distinct lines through the origin,
 whose directions are referred to as the horizontal and the vertical direction, respectively.
By a \emph{coordinate rectangle}, we mean a parallelogram whose edges are parallel to $L_0$, $L_1$.
\emph{Normalized cochains} were defined in Section \ref{s:val-coc} as cochains $\theta$
 that vanish on all vertical lines and on $L_0$.

\begin{lem}
\label{l:ti-norm}
Suppose that the integral of a normalized cochain $\theta$ over the boundary of every coordinate rectangle vanishes.
Then $\theta$ vanishes on every horizontal line.
\end{lem}

\begin{proof}
Consider a rectangle $R$ whose edges are parallel to $L_0$, $L_1$ and with one edge contained in $L_0$.
The integral of $\theta$ over the boundary of $R$ equals zero, since $\omega_\mu=0$.
On the other hand, since $\theta$ is normalized, the integrals of $\theta$ over the vertical edges of $R$ vanish.
For the same reason, the integral over the edge contained in $L_0$, is also zero.
It follows that $\theta$ has zero value on the remaining horizontal edge.
However, the latter can be any horizontal line segment, with a suitable choice of $R$.
Thus, $\theta$ is zero on all horizontal lines, as claimed.
\end{proof}

Suppose now that $\mu$ is a simple translation invariant valuation on plane polygons
 such that the associated 2-form $\omega_\mu$ is identically zero.
The latter requirement (that $\omega_\mu=0$) is equivalent to the property
 that $\mu$ takes zero value of every parallelogram.

\begin{cor}
\label{c:coc-ti}
Suppose that $\theta$ is normalized and $\mu=S\theta$ is translation invariant.
If $\omega_\mu=0$, then $\theta$ is itself translation invariant.
Moreover, translation invariant normalized cochains represent precisely those translation
 invariant simple valuations $\mu$, for which $\omega_\mu=0$.
\end{cor}

\begin{proof}
Indeed, the condition that $\theta$ vanishes on all horizontal and on all vertical lines
 is stable under parallel translations.
The last claim of the corollary is immediate.
\end{proof}

Recall that any cochain $\theta$ can be written as $\sum_{L}\theta^L$, where the summation
 is over all lines through the origin, and $\theta^L$ vanishes on all lines not parallel to $L$.
Translation invariant normalized cochains are precisely those, for which $\theta^{L_0}=\theta^{L_1}=0$.
Also, all $\theta^L$ must be translation invariant.
It follows from Theorem \ref{t:Tinv-1D} that $\theta^L(a,b)=\xi_L(b-a)$ for
 straight line segments $[a;b]$ parallel to $L$, where $\xi_L$ is some additive function on $L$.
This implies the following lemma.

\begin{lem}
\label{l:HG}
Let $L$ be a line through the origin.
Suppose that a translation invariant cochain $\theta^L$ vanishes on all lines not parallel to $L$.
Then $\theta^L=\xi_L\circ\HG_L$, where $\xi_L:\R\to\R$ is a $\Q$-linear function.
\end{lem}

Recall that Hadwiger--Glur valuations $\HG_L$ were defined in Section \ref{s:intro};
 the definition of $\HG_L$ depends on a choice of an orientation of $L$.
Thus, in order to represent $\theta^L$ as $\xi^L\circ\HG_L$, one first needs
 to choose an orientation of $L$; this can be an arbitrary choice for each $L$.

We can now complete the description of all simple translation invariant valuations on plane polygons.
Recall that $\area(P)$ is the area of a polygon $P$; evidently, $\area$ is a simple translation
 invariant valuation.

\begin{thm}
\label{t:stiv}
Let $\mu$ be a simple translation invariant valuation on plane polygons.
Then $\mu$ can be represented as a sum
$$
\mu=\xi\circ\area +\sum_L \xi_L\circ\HG_L,
$$
 where $L$ ranges through all lines through the origin,
 and every line comes with an arbitrarily chosen orientation.
Moreover, if one imposes additionally that $\xi_{L_0}=\xi_{L_1}=0$,
 then this representation is unique.
\end{thm}

Theorem \ref{t:stiv} can be deduced from \cite{HG51} by means of linear algebra
 (an argument to this effect is given in \cite{KKT26}).
Also, this theorem is a special case of Theorem 4.2 from \cite{KP16}.

\begin{proof}
Consider $\omega_\mu$, the associated 2-form of $\mu$.
By Lemma \ref{l:collin}, this is a simple 2-form over $\Q$, which, by Proposition \ref{p:simp2form},
 has the form $(u,v)\mapsto \xi\circ\det(u,v)$.
In other words, for every parallelogram $\Pi$, one has $\mu(\Pi)=\xi\circ\area(\Pi)$.
Note also that $\xi$ is uniquely determined by $\mu$.
Subtracting $\xi\circ\area$ from $\mu$, we may now assume that $\mu$ vanishes on every parallelogram.

Let $\theta$ be the normalized cochain representing $\mu$.
As is established in Section \ref{ss:HG}, the cochain $\theta$ can be uniquely represented as
 the sum of $\theta^L$ over all lines $L$ through the origin, where $\theta^L$ vanishes
 on all lines not parallel to $L$.
Since $\theta$ is normalized, $\theta^{L_0}=\theta^{L_1}=0$.
By Lemma \ref{l:HG}, every $\theta^L$ has the form $\xi_L\circ\HG_L$, for some $\Q$-linear map $\xi_L:\R\to\R$.
\end{proof}

Let $\Ts$ denote the group of all parallel translations of the plane.
Since $\Ts$ acts on polygons, it also acts on valuations;
 translation invariant valuations are exactly those invariant under the action of $\Ts$.
Now build a slightly larger group $\Ts_\pm$ generated by $\Ts$ and the half-turn $z\mapsto -z$.
Elements of $\Ts_\pm$ have the form $z\mapsto a\pm z$, for $a\in\R^2$.
Theorem \ref{t:stiv} implies a characterization of simple $\Ts_\pm$ invariant valuations.

\begin{thm}
  \label{t:Tpm}
Every simple $\Ts_\pm$-invariant valuation on plane polygons has the form $\xi\circ\area$
 for some $\Q$-linear function $\xi:\R\to\R$.
\end{thm}

\begin{proof}
Let $\mu$ be a simple $\Ts_\pm$-invariant valuation.
Fix two lines $L_0$, $L_1$ through the origin, conventionally, the horizontal and the vertical lines.
Since $\mu$ is translation invariant, it can be uniquely represented as
$$
\mu=\xi\circ\area+\sum_L \xi_L\circ\HG_L,
$$
 where $L$ ranges through all lines containing the origin, and $\xi_{L_0}=\xi_{L_1}=0$.
We have chosen and fixed an orientation on every $L$.
Note that $\mu$ is also invariant under the map $z\mapsto -z$.
Applying this transformation to the right-hand side of the formula for $\mu$,
 we obtain that the term $\xi\circ\area$ remains unchanged whereas each of the terms $\xi_L\circ\HG_L$
 changes sign (because $L$ reverses its orientation).
By the uniqueness of the representation, it follows that all $\xi_L=0$.
\end{proof}

\begin{rem}
Curiously, the result is essentially different for a group $\Gs$ generated by $\Ts$
 and some group $\Gs_0$ of pure rotations about $o$ such that the antipodal involution $x\mapsto -x$
 does \emph{not} belong to $\Gs_0$.
In this case, orientations can be consistently chosen on all lines through $o$ from
 the same $\Gs_0$-orbit, and the space of $\Gs$-invariant simple valuations
 is generated by $\area$ together with valuations of the form $\sum_{g\in\Gs_0} \HG_{gL}$.
\end{rem}

\section{Translation scissors congruences}
\label{s:scissor}
Classification of all translation invariant simple valuations implies a solution of
 the translation scissors congruence problem.
Here, we recall this deduction.

\begin{dfn}[Polygon space]
Let $\Vc$ be the vector space over $\Q$, whose generators $[P]$ are labeled by
 compact convex polygons $P\subset\R^2$ with nonempty interior, subject to
 the relations
$$
[P+v]=[P],\quad [P]=[P_+]+[P_-],
$$
 where $v\in\R^2$ is an arbitrary vector, and $P_\pm$ are the two components
 of $P\sm L$, where $L$ is a line cutting through the interior of $P$.
The space $\Vc$ is called \emph{the polygon space}.
\end{dfn}

Higher dimensional versions of the polygon space are considered in
 \cite{JT78,Sah79,Mc89,Mor93,Dup01}.

It is straightforward from this definition that two convex polygons $P$, $Q$
 satisfy $[P]=[Q]$ if and only if there is a (possibly, nonconvex) polygon $R$
 disjoint from $P$ and $Q$ and such that $P\cup R$ and $Q\cup R$ are
 translation scissors congruent.
Observe also that translation invariant simple valuations on polygons can now
 be interpreted as $\R$-valued $\Q$-linear functionals on $\Vc$, i.e.,
 as elements of $\Hom_\Q(\Vc,\R)$.
We want to deduce the following classical theorem (\cite{HG51}) from Theorem \ref{t:stiv}.

\begin{thm}
  \label{t:HG}
Convex polygons $P$, $Q$ are translation scissors congruent
 if and only if  $\area(P)=\area(Q)$ and $\HG_L(P)=\HG_L(Q)$ for
 all lines $L$ through the origin.
\end{thm}

A proof of Theorem \ref{t:HG} involves two steps:
\begin{enumerate}
  \item A theorem of Zylev \cite{Zyl65} claiming that $[P]=[Q]$ if and only if the polygons $P$ and $Q$ are
   translation scissors congruent.
  \item A classical fact of linear algebra, namely, $\zeta(v)=\zeta(w)$ for a pair of vectors $v$, $w\in\Vc$
   and every $\zeta\in\Hom_\Q(\Vc,\R)$ implies $v=w$.
\end{enumerate}
For completeness, we give a proof of Zylev's theorem below
 (in a special case of translation scissors congruence).

\begin{thm}[Zylev]
\label{t:Zyl}
Suppose that plane polygons $P$, $Q$ are disjoint from a polygon $R$ and that
 $P\cup R$ is $\Ts$-scissors congruent to $Q\cup R$.
Then $P$ and $Q$ are $\Ts$-scissors congruent.
\end{thm}

Polygons $P$, $Q$, $R$ are not assumed to be convex, neither are they necessarily connected.

\begin{proof}
Throughout this proof, scissors congruence means $\Ts$-scissors congruence.
By the assumption, both $P\cup R$ and $Q\cup R$ are scissors congruent to the same polygon $X$
 (one can set, e.g., $X:=P\cup R$).
Let $Y$ and $Z$ be subpolygons of $X$ corresponding to $P$ and $Q$.
We know that $Y$ and $Z$ are scissors congruent, via some scissors isomorphism $\psi:Y\to Z$,
 and want to define a scissors isomorphism between $X\sm Y$ and $X\sm Z$.
This is easy if $Y$ and $Z$ are disjoint: define $\phi:X\sm Y\to X\sm Z$ by
 setting $\phi=id$ on $X\sm (Y\cup Z)$ and $\phi=\psi^{-1}$ on $Z$.
Performing the scissors congruence $\phi$ will be referred to as \emph{swapping $Y$ with $Z$ in $X$}.
If $Y$ and $Z$ are not disjoint but have sufficiently small area, then consider
 a scissors isomorphism $Y\cap Z\to W$, where $W\subset X$ is disjoint from $Y\cup Z$; see Fig. \ref{fig:Zyl}.
By the above, $X\sm ((Y\sm Z)\cup W)$ is scissors congruent to $X\sm Z$
 (swap $(Y\setminus Z)\cup W$ with $Z$ in $X$).
On the other hand, $X\sm Y$ is scissors congruent to $X\sm ((Y\sm Z)\cup W)$
 (swap $Y\cap Z$ with $W$ in $X\setminus Y$).
In general, cut $Y$ and $Z$ into sufficiently small (in terms of area) pieces $Y_1$, $\dots$, $Y_n$ and,
 respectively, $Z_1$, $\dots$, $Z_n$ so that $Z_i$ is a translation copy of $Y_i$, for every $i=1$, $\dots$, $n$.
Using the argument just explained, one defines a scissors congruence between
 $X\sm (Y_1\cup \dots\cup Y_k)$ and $X\sm (Z_1\cup\dots \cup Z_k)$ by induction on $k$.
\end{proof}

\begin{figure}
  \centering
  \includegraphics[width=\textwidth]{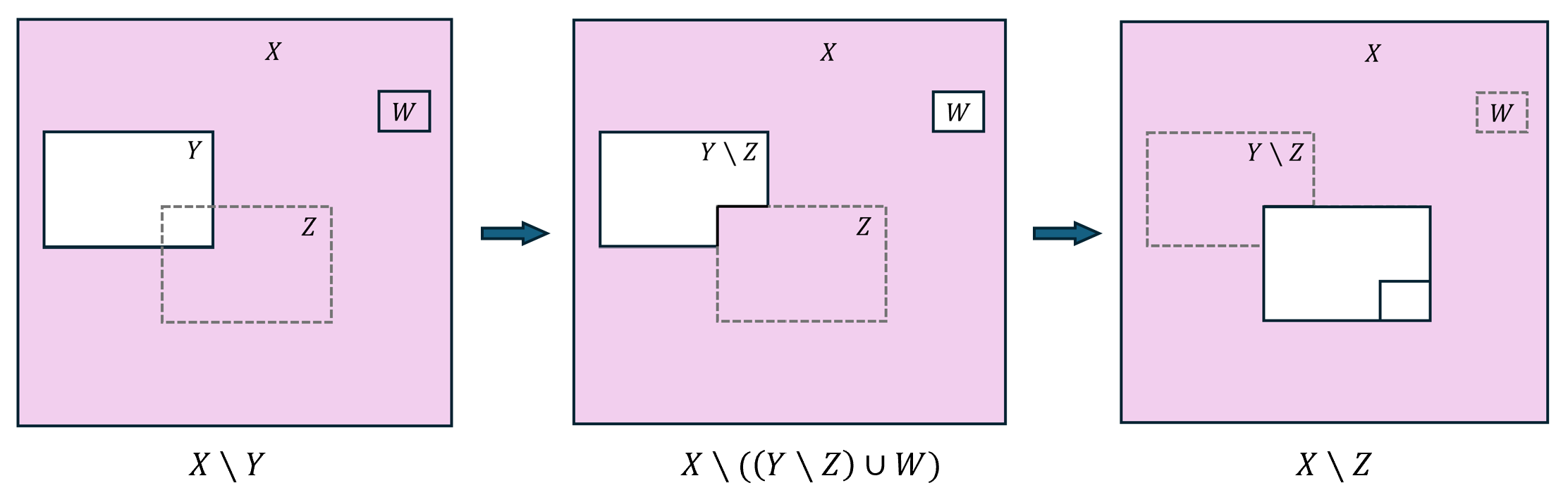}
  \caption{\small A scissors congruence between $X\setminus Y$ and $X\setminus Z$, where $Y$, $Z\subset X$
   are sufficiently small scissors congruent polygons.
   Left arrow: swapping $Y\cap Z$ with $W$ in $X\setminus Y$.
   Right arrow: swapping $(Y\setminus Z)\cup W$ with $Z$ in $X$.}\label{fig:Zyl}
\end{figure}

Zylev's argument is robust under wide variations of context:
 1) instead of $\Ts$-scissors congruence, one can consider $\Gs$-scissors congruence,
 for any group $\Gs$ of area preserving transformations;
 2) polygons do not have to have straight edges (they may be curved), etc.
Sah \cite{Sah79} initiated an axiomatic treatment of scissors congruence problems;
 in particular, Zylev's theorem is proved there in its maximal generality.
Theorem \ref{t:HG} follows from Theorem \ref{t:Zyl} and abstract linear algebra, as explained above.

\section{Polynomial valuations and cochains}
\label{s:poly-coc}
Below, we describe degree $\le d$ polynomial simple valuations on plane polygons.
For brevity, we may write $\deg(\phi)\le d$ to express that a function, or a valuation, $\phi$
 is polynomial of degree $\le d$.

\begin{ex}
Fix a polynomial function $u:\R^2\to\R$ over $\R$, and set $\mu(P)=\int_P u(x,y)\, dx\,dy$;
 then $\mu$ is a polynomial valuation of degree $\le \deg(u)$.
In a certain sense, this is the most important example of polynomial valuations.
\end{ex}

Given a cochain $\theta$, consider the function $\theta_a:x\mapsto \theta(x,x+a)$.
Properties of $\mu=S\theta$ can be translated into properties of this function.
Recall that $L_0$, $L_1$ are the horizontal and the vertical lines through the origin,
 and that $\theta$ being \emph{normalized} means that $\theta_a=0$ for every $a\in L_1$
 and $\theta_a(x)=0$ for $a$, $x\in L_0$.
Finally, recall that $\Delta_a$ is the difference operator: $\Delta_a f(x)=f(x+a)-f(x)$.

\begin{lem}
  \label{l:tha-poly}
Let $\theta$ be a normalized cochain and $\mu=S\theta$.
Then $\deg(\mu)\le d$ if and only if $\theta_{a+b}-\theta_a$ and $\Delta_b\theta_a$ are polynomial functions of degree $\le d$,
 for every $a$ and every vertical $b$.
\end{lem}

\begin{proof}
Assume that vectors $a$ and $b$ are fixed, and that $b$ is vertical.
Note that, up to a sign which depends only on $a$ and $b$, the
 value $\theta_{a+b}(x)-\theta_a(x)$ coincides with $\mu(x+D)$,
 where $D$ is the triangle with vertices $o$, $a$, $a+b$.
The function $\Delta_b\theta_a(x)$ equals, up to a sign, $\mu(x+\Pi(a,b))$.
If $\deg(\mu)\le d$, then both $\mu(x+D)$ and $\mu(x+\Pi(a,b))$
 are degree $\le d$ polynomial functions of $x$.
Conversely, if these two functions have degree $\le d$, for any fixed choice of $a$ and $b$ as above, then
 so does $\mu(x+T)$, where $T$ is any coordinate trapezoid.
Any convex polygon can be represented, up to dimension $<2$ sets, as a union or
 difference of coordinate trapezoids; see the proof of Theorem \ref{t:val-coc}.
It follows that $\mu$ is a polynomial valuation of degree $\le d$.
\end{proof}

We claim that, for a normalized cochain $\theta$
 with $\deg(S\theta)\le d$, the functions $\theta_a$ satisfy $\deg(\theta_a)\le d+1$, for all $a\in\R^2$;
 this is a consequence of Lemma \ref{l:tha-poly} and the following general lemma concerning polynomial functions.

\begin{lem}
\label{l:gen-poly}
Let $\varphi$ be a function on $\R^2$ such that $\deg(\Delta_v\varphi)\le d$,
for every vertical vector $v$ and, moreover, on every horizontal line $\varphi$ is a polynomial of degree $\le d+1$.
Then $\deg(\varphi)\le d+1$.
\end{lem}

\begin{proof}
Let $h$ be a nonzero horizontal vector, and fix some vertical vector $v\ne 0$.
It suffices to establish that $\Delta_v^k\Delta_h^{d+2-k}\varphi=0$, for every $k=0$, $\dots$, $d+2$.
For $k>0$, this follows from the assumption $\deg(\Delta_v\varphi)\le d$.
On the other hand, for $k=0$, this follows from the assumption that $\varphi$
 has degree $\le d+1$ on every horizontal line.
\end{proof}

The following is a characterization of normalized cochains $\theta$ such that $\deg(S\theta)\le d$.

\begin{thm}
\label{t:normcoc-deg}
A normalized cochain $\theta$ satisfies $\deg(S\theta)\le d$ if and only if $\theta$ has the following properties:
\begin{enumerate}
  \item $\deg(\theta_a)\le d+1$ for every $a\in\R^2$;
  \item $\deg(\theta_a-\theta_b)\le d$ whenever $a-b$ is vertical.
\end{enumerate}
\end{thm}

\begin{proof}
Suppose first that $\deg(\mu)\le d$, where $\mu:=S\theta$.
Then $\theta_a$ has degree $\le d$ on every horizontal line.
Indeed, $\theta_a(x)=\pm\mu(T_{x,x+a})$, where $T_{x,x+a}$ is a coordinate trapezoid
 defined in the proof of Theorem \ref{t:val-coc} (strictly speaking, this is true
 only if $[x;x+a]$ is not separated by $L_0$; if it is, one needs to consider two
 degenerate coordinate trapezoids on both sides of $L_0$).
As $x$ moves along a horizontal line, the trapezoid $T_{x,x+a}$ slides by means of parallel translations.
Choose a nonzero vertical vector $v$.
By Lemma \ref{l:tha-poly}, the function $\Delta_v \theta_a$ is polynomial of degree $\le d$.
It follows, by Lemma \ref{l:gen-poly}, that $\deg(\theta_a)\le d+1$, which proves property (1).
Finally, property (2) follows from Lemma \ref{l:tha-poly} again.

Now assume that $\theta$ satisfies properties (1) and (2); we want to show that $\deg(\mu)\le d$.
By Lemma \ref{l:tha-poly}, it suffices to show that $\deg(\Delta_v\theta_a)\le d$,
 for every $a$ and every vertical $v$; this follows from (1).
\end{proof}

A less unique but more invariant representation of $\mu$ is through a
 not necessarily normalized cochain $\theta$ with certain properties.

\begin{thm}
  \label{t:coc-deg}
Every simple degree $\le d$ polynomial valuation $\mu$ can be represented as $S\theta$, where
$$
\deg(\theta_a)\le d+1\quad \hbox{and}\quad\deg(\theta_{a+b}-\theta_a-\theta_b)\le d,
$$
 for all $a$, $b\in \R^2$.
Moreover, such $\theta$ satisfies $S\theta=0$ if and only if $\theta=dg$, for some $g$ with $\deg(g)\le d+2$.
\end{thm}

\begin{proof}
Let $\mu$ be a simple degree $\le d$ polynomial valuation on polygons.
Define $\theta$ as the normalized cochain such that $\mu=S\theta$.
It follows from Theorem \ref{t:normcoc-deg} that $\deg(\theta)\le d+1$.
Note also that
$$
\psi(x):=\theta_{a+b}(x)-\theta_a(x)-\theta_b(x+a)
$$
equals $\pm \mu(x+D)$,
 where $D$ is the triangle with vertices $o$, $a$, $a+b$, and the sign depends only on $a$ and $b$.
Hence, $\deg(\psi)\le d$, which follows from $\deg(\mu)\le d$.
Also, $\theta_b(x+a)-\theta_b(x)$ is a degree $\le d$ polynomial of $x$, since $\theta_b(x)$ is of degree $\le d+1$.
Therefore, $\deg(\theta_{a+b}-\theta_a-\theta_b)\le d$.

Conversely, consider a cochain $\theta$ such that $\theta_a$ is a degree $\le d+1$ polynomial, for every $a$,
 and such that $\deg(\theta_{a+b}-\theta_a-\theta_b)\le d$, for all $a$ and $b$.
Set $\mu=S\theta$.
We need to prove that $\deg(\mu)\le d$.
It suffices to show that, for a triangle $D$ with vertices $o$, $a$, $a+b$, the value
 $\mu(x+D)$ depends on $x$ as a degree $\le d$ polynomial
 (here, we use that every convex polygon can be cut into triangles).
This function of $x$, up to a sign, was denoted by $\psi$ above.
Note that $\psi$ is different from $\theta_{a+b}-\theta_a-\theta_b$ by the additive term
 $x\mapsto \theta_b(x+a)-\theta_b(x)$ of degree $\le d$.
Hence, $\psi$ is indeed a degree $\le d$ polynomial function as claimed,
 which shows that $\mu$ is polynomial of degree $\le d$.

Finally, suppose that $\deg(\theta)\le d+1$ and that $S\theta=0$.
By Theorem \ref{t:val-coh}, the latter means $\theta=dg$, for some function $g:\R^2\to\R$.
Since $\theta_a=\Delta_a g$, it follows that $\deg(\Delta_a g)\le d+1$, for every $a$.
Therefore, we have $\deg(g)\le d+2$, as claimed.
\end{proof}

\section{Polynomial valuations and flags}
\label{s:poly-flag}
We can also interpret the property $\deg(\mu)\le d$ in terms of
 an associated function $f$ on flags such that $\mu=\mu_f$.
Recall that Theorem \ref{t:flag-norm} represents $\mu$ as $\mu_f$ for
 a unique normalized function $f$ on flags.
Fixing a line $L\ni o$, define the function $f:\R^2\to\R$ by the formula $f_L(a):=f(a,L+a)$.
Thus, a function $f$ on flags is represented as a family of functions on $\R^2$,
 parameterized by all lines through the origin.
Note that $f$ being normalized can now be expressed as
$$
f_{L_1}=0,\quad f_{L_0}|_{L_0}=0,\quad f_L|_{L_1}=0,
$$
 for every line $L\ni o$.

\begin{lem}
\label{l:degf}
Suppose that $\deg(\mu)\le d$ and that $f$ is the normalized function on flags with $\mu=\mu_f$.
Then $\deg(f_L)\le d+2$, for every $L$.
\end{lem}

\begin{proof}
Let $\theta$ be the normalized cochain with $\mu=S\theta$.
Recall that $\theta(x,y)=f_L(x)-f_L(y)$, where $L$ is the line through $o$ parallel to the vector $y-x$.
Fix a line $L\ni o$ that is not vertical.
Then $\theta_a=-\Delta_a f_L$, for any $a\in L$.
It follows that $\deg(\Delta_a f_L)\le d+1$.
Also, fix a vertical vector $w\ne 0$ and observe that $\Delta_w f_L(x)$ equals $\pm\mu(x-u_x+\Pi(u_x,w))$,
 where $u_x$ is the vector from $L_1\cap (L+x)$ to $x$.
When $x$ moves along a vertical line, $u_x$ is constant.
Hence, the function $\Delta_w f_L$ is polynomial of degree $\le d$ on every vertical line;
 it follows that $f_L$ has degree $\le d+1$ on every vertical line.
The desired claim $\deg(f_L)\le d+2$ is now easily derived from Lemma \ref{l:gen-poly},
 where $a$ plays the role of $v$ and vertical lines play the role of horizontal lines.
\end{proof}

The following is a complete characterization of normalized functions $f$ on flags
 such that $\deg(\mu_f)\le d$.

\begin{thm}
  \label{t:fpoly}
A normalized function $f$ on flags satisfies $\deg(\mu_f)\le d$ if and only if
 it satisfies the following two properties:
\begin{enumerate}
  \item one has $\deg(f_L)\le d+2$, for every line $L\ni o$;
  \item if $a$, $b\in\R^2$ are such that $a-b$ is vertical, and $L_a$, $L_b$ are lines
  connecting $o$ with $a$, $b$, respectively, then $\Delta_a f_{L_a}-\Delta_b f_{L_b}$
  is a degree $\le d$ polynomial function.
\end{enumerate}
\end{thm}

\begin{proof}
Suppose first that $\deg(\mu_f)\le d$.
Then property (1) holds by Lemma \ref{l:degf}.
Property (2) follows from property (2) of $\theta$ from Theorem \ref{t:normcoc-deg}
 and the observation that the normalized $\theta$ with $S\theta=\mu_f$
 satisfies $\theta_a=-\Delta_a f_{L_a}$ and similarly for $\theta_b$.

Now suppose that properties (1) and (2) of $f$ hold.
Consider the corresponding cochain $\theta$ with $S\theta=\mu_f$.
We have $\theta_a =-\Delta_a f_L$, hence $\deg(\theta_a)\le d+1$;
 this is property (1) of Theorem \ref{t:normcoc-deg}.
Property (2) of $\theta$ (see Theorem \ref{t:normcoc-deg}) follows immediately from property (2) of $f$.
\end{proof}

A \emph{generalized virtual polygon} $P$ (cf. \cite{LMK22,Kho23}) in the plane is defined as a piecewise
 constant function on the complement of an oriented broken line $\d P$, possibly self-intersecting.
The value of this function on a component $U$ of $\R^2\setminus\d P$ equals to
 the index of $\d P$ with respect to any point $a\in U$ (the number of times $\d P$ winds around $a$).
Note that $P$ is determined by $\d P$; also, $\d P$ can be viewed as a chain.
For any simple valuation $\mu$, represent $\mu$ as $S\theta$, for some cochain $\theta$,
 and define $\mu(P)$, for every generalized virtual polygon $P$, as the integral of $\theta$ over $\d P$.
In particular, it makes sense to talk about the area of a generalized virtual polygon.
Recall that two points $a$, $b\in\R^2$ define a \emph{coordinate trapezoid} $T_{a,b}$
 (see the proof of Theorem \ref{t:val-coc} and Fig. \ref{fig:tra1}).
Sometimes, $T_{a,b}$ should be viewed as a generalized virtual polygon, e.g., when $a$ and $b$ are separated by
 $L_0$ or when, as one moves from $a$ to $b$, the interior of $T_{a,b}$ appears on the right rather than on the left.
With proper understanding of $T_{a,b}$ as a generalized virtual polygon, one always has $\theta(a,b)=\mu(T_{a,b})$.

Fix $n$ lines $L_0$, $\dots$, $L_{n-1}$ through the origin so that $L_i$ and $L_{i+1}$ are distinct
 for all $i\in\Z/n\Z$ but non-adjacent lines may coincide.
Also, for each of these lines $L_i$, fix a linear equation $\lambda_i=0$ defining $L_i$,
 where $\la_i:\R^2\to\R$ is some $\R$-linear function.
Now, every $n$-tuple of real numbers $h_0$, $\dots$, $h_{n-1}$ defines
 a unique generalized virtual polygon $P(h_0,\dots,h_{n-1})$ with vertices $a_0$, $\dots$, $a_{n-1}$ (taken in this order),
 where $a_i$ and $a_{i+1}$ lie in the line $\la_i=h_i$, for $i\in\Z/n\Z$.
The following is a corollary of Theorem \ref{t:fpoly}.
It is important to remember that the family of virtual polygons $P(h_0,\dots,h_{n-1})$
 depends not only on $h_0$, $\dots$, $h_{n-1}$ but also on a choice of linear functionals $\lambda_0$, $\dots$, $\lambda_{n-1}$.
Parameters $h_0$, $\dots$, $h_{n-1}$ are called \emph{support numbers}.
A 2D case of \cite[Corollary 5]{PKh92a} (in fact, a slightly more general statement, cf. \cite[Theorem 5.7]{LMK22})
 can now be established by different methods.

\begin{cor}[\cite{PKh92a,LMK22}]
\label{c:fpoly}
Let $P(h_0,\dots,h_{n-1})$ be a generalized virtual polygon with support numbers $h_0$, $\dots$, $h_{n-1}$, as above.
Given any degree $\le d$ polynomial simple valuation $\mu$, the function
 $\mu(P(h_0,\dots,h_{n-1}))$ is a degree $\le d+2$ polynomial function of the support numbers $h_0$, $\dots$, $h_{n-1}$.
\end{cor}

\begin{proof}
Let $f$ be a normalized function on flags such that $\mu=\mu_f$.
By Theorem \ref{t:fpoly}, the function $f$ is polynomial of degree $\le d+2$.
Computing the value of $\mu$ on $P:=P(h_0,\dots,h_{n-1})$ means computing
 the sum of $f(a,L)$ for all flags $(a,L)$ of $P$.
Each term in the sum, $f(a,L)=f_L(a)$ is a degree $\le d+2$ polynomial function of $a$.
It remains only to observe that $a$ itself depends linearly on the support numbers.
\end{proof}

For example, the integral over a polygon $P(h_0,\dots,h_{n-1})$ of a degree $d$ polynomial (over $\R$) function
 defines a degree $d+2$ polynomial of the support numbers $h_0$, $\dots$, $h_{n-1}$.

\begin{rem}
The claim made immediately above has a version dealing with the sum over
 $P(h_0,\dots,h_{n-1})\cap\Z^2$ of the values of a polynomial function $\alpha:\Z^2\to\R$;
 the latter claim is also covered by \cite{PKh92a}.
Even though this latter version can also be obtained by the methods of this paper,
 a number of modifications are needed.
One has to account for the following factors: (1) the sum of $\alpha(x)$ over $x\in\Z^2\cap P$
 is a non-simple valuation, (2) it is polynomial only under translations by $\Z^2$.
\end{rem}

Recall that the \emph{Minkowski sum} of two sets $A$, $B\subset\R^2$ is defined as
$$
A+B:=\{a+b\mid a\in A,\ b\in B\}.
$$
This depends on an identification between $\R^2$ and a real 2D vector space,
 that is, on the choice of the origin in $\R^2$.
If the linear functionals $\la_0$, $\dots$, $\la_{n-1}$ are fixed, then
$$
P(h_0,\dots,h_{n-1})+P(h'_0,\dots,h'_{n-1})=P(h_0+h'_0,\dots,h_{n-1}+h'_{n-1}),
$$
 whenever all polygons involved are genuine (rather than generalized virtual) convex polygons
  and all $h_i$, $h'_i$ are positive.
Minkowski addition can be extended to generalized virtual polygons so that to keep the above general formula;
 it endows the set of all generalized virtual polygons with the structure of an abelian group.
Corollary \ref{c:fpoly} can now be interpreted as a degree $\le d+2$ polynomial
 dependence of $\mu(P)$ on $P$ \emph{with respect to Minkowski addition}.

\begin{rem}
Polynomial functions on a commutative semigroup admit canonical extensions to
 the corresponding Grothendieck group, see \cite[Corollary 2.13]{Kho25} for an explicit formula.
On the other hand, convex polygons in the plane form a commutative semigroup
 under Minkowski addition whose Grothendieck group identifies with the group of
 \emph{virtual polygons}, that is, generalized virtual polygons of the form 
  $P(h_0-h'_0,\dots,h_{n-1}-h'_{n-1})$, where $P(h_0,\dots,h_{n-1})$ and $P(h'_0,\dots,h'_{n-1})$
 are convex while $h_i$, $h'_i>0$ for $i=0$, $\dots$, $n-1$.
Thus, it is enough to know the values of a polynomial valuation on all convex
 polygons to explicitly recover its values on all virtual polygons.
\end{rem}

\section{A description of polynomial valuations}
\label{s:desc-poly}
Our starting point will be Theorem \ref{t:coc-deg}.
According to this theorem, any degree $\le d$ polynomial valuation is obtained
 as $S\theta$ for some degree $\le d+1$ polynomial cochain.
If $\deg(\theta)\le d$, then $\deg(S\theta)\le d$ automatically, no additional assumptions on $\theta$ are needed
 in this case.
Thus, it suffices to consider $\theta$ modulo degree $\le d$ polynomial cochains.
Recall that any polynomial function $f:\R^2\to\R$ over $\Q$ is a sum of its homogeneous components;
 write $f^{[m]}$ for the degree $m$ component of $f$; it satisfies
 $f^{[m]}(\la x)=\la^m f^{[m]}(x)$ for all $\la\in\Q$ and $x\in\R^2$.
Taylor's formula allows to define $D_af$, the \emph{directional derivative of $f$ along a vector} $a\in\R^2$
 even if $f$ is a discontinuous polynomial function.
Namely, set $D_a=\log(1+\Delta_a)$ so that to have $\Delta_a+1=e^{\Delta_a}$ (Taylor's formula),
 where $\Delta_af(x):=f(x+a)-f(x)$; the operator series
$$
\log(1+\Delta_a)=\Delta_a-\frac{\Delta_a^2}{2}+\frac{\Delta_a^3}{3}\mp \dots
$$
when evaluated on any polynomial function, yields only finitely many nonzero terms.
The derivative of a homogeneous polynomial is again a homogeneous polynomial, whose degree drops by one.

By property (2) of Theorem \ref{t:coc-deg}, the polynomial $\theta_{a+b}-\theta_a-\theta_b$
 has degree $\le d$, for every $a$, $b\in\R^2$.
It follows that $\theta^{[d+1]}_{a+b}-\theta^{[d+1]}_a-\theta^{[d+1]}_b$ vanishes identically.
Hence, $\theta$ defines a polynomial 1-form
$\omega^\theta(a,x)=\theta^{[d+1]}_a(x)$.
Recall that a \emph{degree $\le d$ polynomial 1-form} on $\R^2$ is defined as
 a function $\omega:\R^2\times\R^2\to\R$ that depends linearly (over $\Q$)
 on the first argument and polynomially of degree $\le d$ (also over $\Q$) on the second argument.
The 1-form $\omega^\theta$ defined above is said to be \emph{associated with the cochain $\theta$}.
Note that $\omega^\theta(a,x)$ is a homogeneous degree $d+1$ polynomial of $x$,
 for every $a\in\R^2$.

Consider a polynomial 1-form $\omega$.
For every $a\in\R^2$, define a polynomial function $\omega_a:\R^2\to\R$ by the formula
 $\omega_a(x)=\omega(a,x)$; every polynomial 1-form can be thought of as a
 collection of these functions depending linearly on $a$.
Say that $\omega$ is \emph{exact on the lines} if,
 for every line $L$ through the origin, there is a polynomial function $F_L:\R^2\to\R$,
 called the \emph{$L$-antiderivative of $\omega$},
 such that $\omega_a(x)=D_a F_L(x)$ for all $a\in L$.
One can easily verify that $\omega$ is exact on the lines if and only if
 $D_a\omega_b=D_b\omega_a$ for every pair of vectors $a$, $b$ that are
 linearly dependent over $\R$.

\begin{lem}
  \label{l:ex-on-lines}
Let $\theta$ be a normalized degree $d+1$ cochain satisfying properties (1)--(2)
 of Theorem \ref{t:coc-deg}.
Then the associated 1-form $\omega^\theta$ is exact on the lines.
\end{lem}

\begin{proof}
By Theorem \ref{t:coc-deg}, the valuation $\mu=S\theta$ is polynomial of degree $\le d$.
Theorem \ref{t:fpoly} now implies that the corresponding normalized function $f$ on flags
 is such that $\deg(f_L)\le d+2$, for every line $L$ though the origin.
Recall that $\theta_a=-\Delta_a f_L$, for every $a\in L$, and that $D_af_L$ coincides
 with the top degree homogeneous component of $\Delta_a f_L$.
It follows that
$$
\omega^\theta_a=\theta_a^{[d+1]}=-(\Delta_a f_L)^{[d+1]}=-D_a f^{[d+2]}_L.
$$
We see that $\omega^\theta_a$ is exact on the lines, as claimed.
\end{proof}

If a polynomial 1-form $\omega$ is exact on the lines, then one can talk about
 the integral of $\omega$ over any 1-chain.
For an oriented interval $[a;b]$ parallel to a line $L\ni o$, set $\int_{[a;b]}\omega=F_L(b)-F_L(a)$,
 where $F_L$ is an $L$-antiderivative of $\omega$.
Clearly, this integral does not depend on the choice of $F_L$.

\begin{thm}
  \label{t:polyval-descr}
Any degree $\le d$ polynomial valuation $\mu$ has the form $\mu^{\bullet}+\mu^{\circ}$,
 where $\mu^{\bullet}(P)$ is the integral over $\d P$ of some degree $d+1$ homogeneous
 polynomial 1-form that is exact on the lines, and $\mu^{\circ}$ is the integral over $\d P$ of some degree $d$
 polynomial cochain.
\end{thm}

If $\mu(P)$ depends on $P$ continuously in the sense of the Hausdorff metric,
 it can be shown then $\mu^{\bullet}(P)$ can be represented as the integral over $\d P$ of some
 degree $\le d+1$ polynomial 1-form over $\R$, hence, by the Stokes theorem,
 as the integral over $P$ of some degree $\le d$ polynomial 2-form.
Thus, we may think of $\mu^{\bullet}$ as an interior integral, and of $\mu^{\circ}$ as
 a boundary integral, even though, in practice, both $\mu^\bullet(P)$ and $\mu^\circ(P)$
 are computed by summing up certain numbers associated to the vertices of $P$.
One can further represent $\mu^{\circ}$ (with no continuity assumptions as above)
 as a (possibly, infinite) linear combination of $S\theta^L$, where $\theta^L$
 is a degree $d$ polynomial cochain 
 vanishing on every line not parallel to $L$.

\begin{proof}[Proof of Theorem \ref{t:polyval-descr}]
Assume that $L_0$, $L_1$ are fixed, and consider the normalized cochain $\theta$ with $\mu=S\theta$.
By Theorem \ref{t:coc-deg}, one has $\deg(\theta)\le d+1$, and $\deg(\theta_{a+b}-\theta_a-\theta_b)\le d$,
 for all $a$, $b\in\R^2$.
Let $x\mapsto \omega(a,x)$ be the degree $d+1$ homogeneous part of $\theta_a$;
 the polynomial 1-form $\omega$ is exact on the lines, by Lemma \ref{l:ex-on-lines}.
In particular, the integral of $\omega$ over $\d P$ makes sense.
Set $\mu^\bullet(P)$ to be this integral.
In fact, $\mu^{\bullet}(P)$ is the integral of some degree $\le d+1$ polynomial cochain $\theta^{\bullet}$ over $\d P$
 such that $\omega_a=D_a f_L$ for a polynomial function $f$ with $\theta^\bullet_a=\Delta_a f$;
 cf. the proof of Lemma \ref{l:ex-on-lines}.

Cochains $\theta$ and $\theta^\bullet$ are, in general, different.
On the other hand, we claim that $\deg(\theta-\theta^\bullet)\le d$.
Indeed, $D_a f_L=\theta_a^{[d+1]}$ is the degree $d+1$ homogeneous part of $\Delta_a f=\theta^\bullet_a$.
Setting $\theta^\circ:=\theta-\theta^\bullet$, one can write
 $\mu$ as $\mu^\bullet+\mu^\circ$, where $\mu^\bullet=S\theta^\bullet$ and $\mu^\circ=S\theta^\circ$.
Note also that $\theta^\bullet$ is obtained by integrating $\omega$ that is a degree $d+1$ homogeneous
 polynomial 1-form exact on the lines.
\end{proof}


\end{document}